\documentclass[11pt,a4paper]{article}

\usepackage{graphicx}
\usepackage{amsmath}
\usepackage{amsfonts}
\usepackage{amssymb}
\usepackage{enumerate}
\usepackage{lscape}
\usepackage{longtable}
\usepackage{rotating}
\usepackage{multirow}
\usepackage{color}
\usepackage{url}
\usepackage{subfigure}
\usepackage{rotating}
\newtheorem{theorem}{Theorem}[section]

\usepackage[ruled,vlined,noline,linesnumbered]{algorithm2e}

\newtheorem{lemma}{Lemma}[section]

\newtheorem{remark}{Remark}[section]

\newtheorem{assumption}{Assumption}[section]

\newenvironment{proof}[1][Proof]{\textbf{#1.} }{\ \rule{0.5em}{0.5em} \vspace{1ex}}
\usepackage[ruled,vlined,noline,linesnumbered]{algorithm2e}
\usepackage{algorithmic}
\def\real{\mathbb{R}}

\newcommand{\ignore}[1]{{}}

\begin{document}

\title{Complexity iteration analysis for strongly convex multi-objective optimization using a Newton path-following procedure}
\author{
E. Bergou \thanks{King  Abdullah  University  of  Science  and  Technology  (KAUST),  Thuwal,  Saudi  Arabia. MaIAGE, INRAE, Universit\'e Paris-Saclay, 78350 Jouy-en-Josas, France
 ({\tt elhoucine.bergou@inra.fr}). This author received support from the AgreenSkills+ fellowship programme which has received funding from the EU's Seventh Framework Programme under grant agreement No FP7-609398 (AgreenSkills+ contract).
}
\and
Y. Diouane\thanks{Institut Sup\'erieur de l'A\'eronautique et de l'Espace (ISAE-SUPAERO), Universit\'e de Toulouse, 31055 Toulouse Cedex 4, France
 ({\tt youssef.diouane@isae.fr}).
}
\and
V. Kungurtsev \thanks{Department of Computer Science, Faculty of Electrical Engineering, Czech Technical University in Prague. ({\tt vyacheslav.kungurtsev@fel.cvut.cz}).
Research supported by the OP VVV project
CZ.02.1.01/0.0/0.0/16\_019/0000765 Research Center for Informatics.
}
}
\maketitle
\footnotesep=0.4cm
{\small
\begin{abstract}
In this note we consider the iteration complexity of solving strongly convex multi objective optimization.
We discuss the precise meaning of this problem, and indicate it is loosely defined, but the most natural
notion is to find a set of Pareto optimal points across a grid of scalarized problems. We derive that
in most cases, performing sensitivity based path-following after obtaining one solution
is the optimal strategy for this task in terms of iteration complexity.
\end{abstract}

\bigskip

\begin{center}
\end{center}
}

\section{Introduction}
Consider the following multi-objective optimization problem:
\begin{equation}\label{eq:moo}
\min_{x \in \mathbb{R}^n} f(x),
\end{equation}
 where $f: \mathbb{R}^n \rightarrow \mathbb{R}^m$ is a  strongly convex and twice continuously differentiable function. Our target is to find \emph{weak Pareto-optimality} points for the problem (\ref{eq:moo}), recalling that \emph{weak Pareto-optimality} holds at point $\tilde x$ if for all $d \in \mathbb{R}^n$, there exists an $i \in \{1,\ldots,m\}$ such that
\[
\nabla f_i(\tilde x)^{\top} d \ge 0.
\]

For single objective optimization, \emph{worst case iteration complexity} quantifies the number of iterations
that could be necessary, in the worst case (i.e., for the most ill-behaved problem), for an algorithm to achieve
a certain level of satisfaction of an approximate measure of optimality, typically a small norm for the
gradient~\cite{nesterov2013introductory}. Classically, the multi objective optimization community had
not considered attempting to derive bounds on iteration complexity for problems in vector optimization.
In~\cite{fliege2018complexity,calderon2018complexity} the complexity of gradient descent for multi-objective optimization was considered. 
Rates were derived for obtaining some point satisfying weak Pareto-optimality.
However, in deriving the complexity result, convergence of the algorithm to some Pareto optimal point is assumed.
More fundamentally, though, consider the so-called scalarized problem, parametrized by $\{\lambda_i\}_{i=1,..,m}$ 
\begin{equation}\label{eq:scalarmoo}
\min_{x\in\mathbb{R}^n} \sum_{i=1}^m \lambda_i f_i(x),
\end{equation}
for any $\{\lambda_i\}_{i=1,..,m}$ satisfying $0<\lambda_i<1,\,\sum_i \lambda_i=1$. A stationary point of this problem is also Pareto optimal for~\eqref{eq:moo}. Thus, one can
find a Pareto optimal point, at least for strongly convex multi-objective problems, by simply choosing any 
arbitrary convex combination $\{\lambda_i\}_{i=1,..,m}$ and solving the resulting mono-objective problem, 
thus the worst case iteration complexity of 
finding \emph{some} Pareto optimal point is already a known problem, it corresponds to the worst case
iteration complexity of solving a single objective strongly convex optimization problem.

In the multi-objective optimization literature, e.g.,~\cite{marler2004survey}, scalarization is typically,
at most, a step in the process of finding the solution of a multi-objective problem, where the 
definition of solution depends on the context. In particular, it can be that the goal of the
optimization is 1) tracing the Pareto front itself, so in some sense finding all, or some adequate 
approximation to all, stationary points, or 2) find an appropriately \emph{best} point of the Pareto front 
through some secondary metrics, or using an interactive environment with a human participant who grades
potential solutions.

In this note, we shall concern ourselves with the first task: establish complexity bounds for some appropriate notion
of finding the entire Pareto front. To this end we define the problem as, for all $\lambda\in \Lambda \subset \mathbb{R}^m$, find
\begin{equation}\label{eq:Pareto}
\min_{x\in\mathbb{R}^n} \sum_{i=1}^m \lambda_i f_i(x),
\end{equation}
where $\Lambda $ is some finite grid of elements $\lambda$ satisfying $0<\lambda_i<1,\,\sum_i \lambda_i=1$. % and $d$ denoting the maximum width of any two neighbors on the grid, that is $$d = \min_{(\lambda, \lambda')\in \Lambda~,~\lambda \ne \lambda'} \left(\max_{1\le i \le m }|\lambda_i-\lambda'_j|\right).$$
Given the constraint on the sum, we can consider $m-1$ dimensions as free 
which in turn entirely determine the remaining $\lambda_i$. We thus divide each side of the hypercube $[0,1]$ by some desired
width of the grid $d$, and thus there are $\left\lfloor{\frac{1}{d}}\right\rfloor^{m-1}$ total possible grid points,
where $\lfloor{a}\rfloor$ denotes the greatest integer less than or equal to $a$. Conversely, 
since they form a rectangular lattice, we can define the quantity $d$ denoting the maximum width of any two neighbors on the grid, that is $$d = \min_{(\lambda, \lambda')\in \Lambda~,~\lambda \ne \lambda'} \left(\max_{1\le i \le m }|\lambda_i-\lambda'_i|\right).$$

We organize this paper as follows. In section \ref{sec:Npathfoll}, we describe our algorithm to solve the multi objective strongly convex problem. We explain how to use a Newton path following procedure to find the entire Pareto front. 
Section \ref{sec:PNA} addresses the convergence of our algorithm by   characterizing its Complexity. A numerical illustration about the efficiency of our approach is given in section \ref{sec:Experiments}. 
 Conclusions  are given in section \ref{sec:Conclusion}.  

%During all this paper, we will assume strong convexity of all the components of f. For that, for all $i \in \{1,\ldots,m \}$, we will assume that there exits a positive constant $c_i>0$, such that the
%following known inequality
%\begin{equation}\label{asump:f}
%f_i(x) \le f_i(y) +\nabla f_i(x)^{\top}(x-y) - \frac{c_i}{2}\|x-y\|^2
%\end{equation}
%is valide for all $x,y \in \real^n$. 
\section{Pathfollowing for finding the entire Pareto front.} \label{sec:Npathfoll}
%\subsection{Strongly Convex Optimization}
%We shall consider the strongly convex case for illustrative simplicity. We recall from, e.g.,~\cite{bubeck2015convex}
%that using only first order information at best one can solve a strongly convex optimization problem linearly,
%\begin{theorem}~\cite[Theorem 3.18]{bubeck2015convex}
%For a $\beta$-smooth and $\alpha$-strongly convex function $f$ with condition number $\kappa = \frac{\beta}{\alpha}$, 
%the accelerated gradient descent iteration satisfies,
%\[
%f(x_t) -f(\tilde x)\le \frac{\alpha+\beta}{2}\|x_0-\tilde x\|e^{-\frac{t-1}{\sqrt{\kappa}}},
%\]
%\end{theorem}
%where $t$ is the iterations count.
%
%
%Now by the Lipschitzian gradient property it holds that,
%\begin{equation}\label{eq:convexnabla}
%\|\nabla f(x_t)\| \le L\frac{\alpha+\beta}{2}\|x_0-\tilde x\|e^{-\frac{t-1}{\sqrt{\kappa}}}
%\end{equation}
%Thus in order to obtain a $\epsilon$ distance to the solution, $O(\log(1/\epsilon))$ iterations
%must be taken, with each iteration involving the computation of one gradient. As a result,
%naively, one can obtain the entire Pareto front by solving each of the $\left\lfloor{\frac{1}{d}}\right\rfloor^{m-1}$ scalarized
%problems defined across the grid points independently
%with accelerated gradient descent, to obtain an overall complexity of 
%$O\left(\log(1/\epsilon)\left\lfloor{\frac{1}{d}}\right\rfloor^{m-1}\right)$.

Recall that for a fixed $\lambda^0 \in \Lambda$, using only first order information one can solve a strongly convex optimization problem of the type (\ref{eq:Pareto})
at best linearly (see for instance ~\cite[Theorem 3.18]{bubeck2015convex}, using a gradient descent based method). Namely, in order to obtain $\epsilon$ distance to the solution of  (\ref{eq:Pareto}) for a fixed $\lambda^0 \in \Lambda$, $\mathcal{O}(\log(1/\epsilon))$ iterations must be taken, with each iteration involving the computation of one gradient. As a result,
naively, one can obtain the entire Pareto front by solving each of the $\left\lfloor{\frac{1}{d}}\right\rfloor^{m-1}$ scalarized
problems defined across the grid points independently
with a gradient descent, to obtain an overall complexity of  $\mathcal{O}\left(\log(1/\epsilon)\left\lfloor{\frac{1}{d}}\right\rfloor^{m-1}\right)$.

In this note, we will propose finding the entire Pareto front by performing path-following using the implicit function theorem. Later we will show that the proposed strategy will reduce the overall iteration complexity drastically relative to naively solving every scalarized problem separately.
To start with, for $\lambda^{(0)} \in \Lambda$ we obtain the solution $x^{(0)} \in \mathbb{R}^n$ of the problem  (\ref{eq:Pareto}), (for instance,  by using a gradient descent method using the following stopping criterion  $\left\|\sum_{j=1}^m \lambda^{(0)}_j \nabla f_i(x)\right\| \le \epsilon$). Note that such point $x^{(0)}$ gives the Pareto optimal point of the problem (\ref{eq:moo}) associated to $\lambda^{(0)}$. Now, let $\lambda^{(1)}$ be one of the closest neighbors of $\lambda^{(0)} $ in the finite grid $\Lambda$, our goal is to apply a predictor-corrector scheme to deduce a new $x^{(1)}$ corresponding to 
the Pareto optimal solution of the problem (\ref{eq:moo}) associated with $\lambda^{(1)}$. 

Pathfollowing, or tracing a set of solutions for a parametrized nonlinear system of equations across a range of parameters, is an important
algorithmic tool, for which an introduction can be found in~\cite{allgower2003introduction}. Closest to our work,
a predictor-corrector pathfollowing procedure for strongly convex optimization problems (interpreted as strongly regular variational inequalities)
is given in~\cite{dontchev2013euler}. In this work it is shown that for this parametric problem a property of uniform strong regularity holds
and a procedure involving one tangential predictor (Euler) and one corrector (Newton) step result in a series of iterates whose
distance to a set of solutions to the parametric variational inequality is of the order of $d^4$,  %\ebnote{ $h$ is not equal to $d$?}
 where $d$ is, in this context, the grid spacing. Thus
there exists $C$ such that if $d\le C(\epsilon)^{1/4}$, a set of solutions with approximate optimality $\epsilon$ across a set of parameters
can be found. If applied to the multiobjective Pareto front context, the number of Euler-Newton continuation
steps is the number of grid points, which corresponds to $d^{-1}\ge C^{-1}\epsilon^{-1/4}$. If the desired grid is already small enough, then it is clear
that this pathfollowing procedure outperforms the naive method of solving the standalone problem at every grid point. Otherwise, it depends
on the magnitude of the desired number of additional grid points required to perform pathfollowing.

We consider an alternative predictor-corrector scheme that is more aggressive in its use of potentially longer tangential steps and 
multiple Newton iterations. In particular, this is more suitable for obtaining the set of solutions across the Pareto front with
the tightest iteration complexity bound. This predictor-corrector procedure will be repeated until we handle all the elements from $\Lambda$.
A formal description of the algorithm is given as Algorithm \ref{alg:PNA}.

\begin{algorithm}[t]
	\caption{A Newton path-following procedure for the entire Pareto front.}
	\label{alg:PNA}
	\begin{algorithmic}
	\STATE{\textbf{Input:}} Let $\Lambda=\{\lambda^{(0)},\ldots, \lambda^{(p-1)}\} \subset \real^m_{+}$ be some finite grid of $p$ elements satisfying, for all $j=0,\ldots,p-1$,  $\sum_i^m \lambda^{(j)}_i=1$, $\lambda^{({j+1})}$ is one of the closest neighbors not yet visited of $\lambda^{(j)}$.
	\STATE{\textbf{Output:}} The entire Pareto front by performing path-following associated with $\Lambda$: $x^{(0)}, x^{(1)}, \ldots, x^{(p-1)}$.
\vspace{3mm}
\STATE{\textbf{Compute an initial Pareto optimal point $x^{(0)}$}}, i.e., 
\begin{equation} \label{initial:pareto}
x^{(0)}= \arg \min_x f^{(0)}(x), \, \, \text{where }~~ f^{(0)}(x) = \sum_{i=1}^m \lambda_i^{(0)} f_i(x).
\end{equation}
Set $k=0$.
\STATE{\textbf{Step 1: Compute a predictor $\bar x^{(k+1)}$}}, i.e., 
\begin{equation}\label{predictor:step}
\bar x^{(k+1)} =x^{(k)}+  \left[\sum_{j=1}^m \lambda_j^{(k)} \nabla^2 f_j\left(x^{(k)}\right)\right]^{-1} \left(\sum_{i=1}^m \left(\lambda_i^{(k+1)}- \lambda_i^{(k)}\right) \nabla f_i\left(x^{(k)}\right)\right)
\end{equation}
\STATE{\textbf{Step 2: Apply a Newton correction to compute $x^{(k+1)}$}}, i.e.,  starting from $\bar x^{(k+1)}$ run the Newton method to solve 
\begin{equation}
x^{(k+1)} = \arg \min_x f^{(k)}(x), \, \, \text{where }~~f^{(k)}(x) = \sum_{i=1}^m \lambda_i^{(k+1)} f_i(x).
\end{equation}
%We denote by $x^{(k+1)}$ the obtained solution. 

\textbf{If} $k=p-1$ \textbf{then} Stop, \textbf{otherwise} increment $k$ by $1$ and go to  \textbf{Step 1}.
\end{algorithmic}
\end{algorithm}
%
%\subsection{path-following}
%Consider, however, the possibility of performing path-following using the implicit function theorem. 

The predictor step (Step 1 of Algorithm \ref{alg:PNA} ) is formulated using the implicit function theorem. Therefore, first, let's recall the implicit function theorem adapted to our context.
\begin{theorem}\label{th:implicitfun} Let $g: \mathbb{R}^{n+m} \rightarrow \mathbb{R}^n$ be a continuously differentiable function
for a parametrized system of equations,
\[
g(x,\lambda)=0, ~~~\mbox{where}~~x\in \mathbb{R}^n ~\mbox{and}~ \lambda\in \mathbb{R}^m.
\]
Consider that there exists a solution satisfying $g(x_0,\lambda_0)=0$. If the Jacobian matrix $J_{g,x}(x_0,\lambda_0)$ of $g$ with respect to $x$ is invertible, then there exists an open
neighborhood $\mathcal{B} \subset \mathbb{R}^m $ such that there exists a unique continuously differentiable path $\tilde x(\lambda)$ defined on $\lambda \in \mathcal{B}$ with $\tilde x(\lambda_0)=x_0$
and $g(\tilde x(\lambda),\lambda)=0$ for all $\lambda \in \mathcal{B}$. Furthermore, it holds that the derivative of $\tilde x(\lambda)$ over $\mathcal{B}$ is given by
\begin{equation}\label{eq:impfunder}
 \frac{\partial \tilde x}{\partial \lambda}(\lambda) = -\left[J_{g,x}(\tilde x(\lambda),\lambda)\right]^{-1} \frac{\partial g}{\partial \lambda}(\tilde x(\lambda),\lambda).
\end{equation}
\end{theorem}
We consider applying Theorem \ref{th:implicitfun} to the optimality conditions of~\eqref{eq:scalarmoo} given by the following parametrized system of equations
\[
g(x,\lambda) = \sum_{i=1}^m \lambda_i \nabla f_i(x) = 0.
\]
Precisely, for a given iteration index $k$, consider that we have a solution $x^{(k)}$ at some $\lambda^{(k)} \in \Lambda$, i.e., 
\begin{equation}\label{epsilon:solution}
 \sum_{j=1}^m \lambda^{(k)}_j \nabla f_i(x^{(k)})=0.
\end{equation}
Since we consider strongly convex objectives, it holds that the matrix
\[
\sum_{i=1}^m \lambda^{(k)}_i \nabla^2 f_i(x)
\]
is invertible for all $x\in  \mathbb{R}^n$, in particular the inverse norm is bounded by the inverse of the weighted sum of the
strong convexity constants of the problem. Thus by Theorem~\ref{th:implicitfun} we have
that there exists a unique path $\tilde x^{(k)}(\lambda)$ for any choice of $\lambda \in \mathcal{B}_k$ such that $\tilde x^{(k)}(\lambda^{(k)})=x^{(k)}$ and 
$\sum_{i=1}^m \lambda_i \nabla f_i(\tilde x^{(k)}(\lambda)) = 0$, for some ball $\mathcal{B}^{(k)}$ around $\lambda^{(k)}$.
Furthermore, the derivative of the path given by~\eqref{eq:impfunder} is defined for all $\lambda \in \mathcal{B}^{(k)}$ to satisfy,
\begin{equation*}
 \frac{\partial \tilde x^{(k)}}{\partial \lambda}(\lambda) = -\left[\sum_{j=1}^m \lambda_j \nabla^2 f_j\left(\tilde x^{(k)} (\lambda)\right)\right]^{-1} \left[\nabla f_1\left(\tilde x^{(k)}(\lambda)\right), \ldots, \nabla f_m \left(\tilde x^{(k)}(\lambda)\right) \right].
\end{equation*}
Consider now a Taylor
expansion of $\tilde x^{(k)}(\lambda)$ for all $\lambda \in \mathcal{B}^{(k)}$, this is given by
\begin{equation}\label{eq:derivpath}
\tilde x^{(k)}(\lambda) = x^{(k)} -  \left[\sum_{j=1}^m \lambda^{(k)}_j \nabla^2 f_j\left((x^{(k)}\right)\right]^{-1}  \sum_{i=1}^m \left(\lambda_i-\lambda^{(k)}_i\right) 
\nabla f_i\left(x^{(k)}\right) +\mathcal{O}\left(\left\|\lambda-\lambda^{(k)}\right\|^2\right)
\end{equation}
Motivated by the discussion on Newton's method applied for path-following in~\cite[Chapter 5]{deuflhard2011newton}, for all $\lambda \in \mathcal{B}^{(k)}$, we define a predictor $\bar x^{(k)} (\lambda)$ by computing 
\begin{equation}\label{eq:derivpath2}
\bar x^{(k)} (\lambda) = x^{(k)} -  \left[\sum_{j=1}^m \lambda^{(k)}_j \nabla^2 f_j\left((x^{(k)}\right)\right]^{-1}  \sum_{i=1}^m \left(\lambda_i-\lambda^{(k)}_i\right) 
\nabla f_i\left(x^{(k)}\right) 
\end{equation}
which is precisely 
the ``tangent continuation method''  with the order $p=2$ as given in~\cite[Page 239]{deuflhard2011newton}. 
We define the remainder term as $\eta^{(k)}$. 

%This term can be bounded 
%by the maximum
%of the second derivative of the path, i.e.,
%\[
%\eta^{(k)} \le \sup_\lambda \left\|\frac{\partial^2 \tilde x^{(k)}}{\partial \lambda^2}(\lambda) \right\|\left\|\lambda-\lambda^{(k)}\right\|^2.
%\]

%The remainder term in in~\eqref{eq:derivpath} is bounded by,
%\[
%\frac{\eta^{(k)}}{2}\left\|\lambda-\lambda^{(k)}\right\|^2
%\]

Assuming that $\lambda^{(k+1)}$ is close enough to $\lambda^{(k)}$ (i.e., $\lambda^{(k+1)} \in \mathcal{B}^{(k)}$), the predictor step $\bar x^{(k+1)}$ given by (\ref{predictor:step}) in Algorithm \ref{alg:PNA}  is defined as follows
$$
\bar x^{(k+1)}= \bar x^{(k)} \left(\lambda^{k+1}\right).
$$
There is a remaining algorithmic necessity before this becomes practical as the predictor $\bar x^{(k+1)}$ does not  necessarily satisfy the desired level of stationarity. 
To achieve a point closer to the actual solution, we  consider a ``corrector'' step $x^{(k+1)}$ by using a form of the
Newton step. For a particular level of error, the ordinary Newton method is quadratically convergent
towards the solution starting from the predicted point if the original point is sufficiently close. Thus we 
require that the predictor $\bar x^{(k+1)}$ is sufficiently accurate
and determine the size of the step $\lambda^{(k+1)}-\lambda^{(k)}$ appropriately.

\section{Characterizing the Complexity of Algorithm \ref{alg:PNA}} 
\label{sec:PNA}

Based on the ideas above, we can consider iteration complexity in a new sense. For a given iteration $k$,
consider having a point solution for a particular $\lambda^{(k)}$, up to an optimality tolerance
with a desired $\epsilon$. Then consider path-following from $\lambda^{(k)}$ to some $\lambda^{(k+1)}$
where $\lambda^{(k+1)} - \lambda^{(k)}$ is small enough (in terms of desired grid-spacing $d$) to be able to determine the associated solution on the Pareto front. The same procedure is repeated across
all the grid $\Lambda$ until all solutions of the Pareto front have been found.

% Notice that given one solution
%a solution to the next grid point is just a predictor step and a Newton step, provided that the Newton step
%is able to reduce the residual by the error noted above. %, $\eta^{(k)}\left \|\lambda^{(k+1)} - \lambda^{(k)}\right\|^2$.

Before developing our complexity analysis, we state formally our working assumptions on the objective function $f$.
\begin{assumption} \label{asm:1}
 The objective function $f: \mathbb{R}^n \rightarrow \mathbb{R}^m$ is twice continuously differentiable and strongly convex.
In particular, there exists two positive constant $c>0$ and $L>0$, for all $i \in \{1,\ldots,m \}$ and $y \in \real^n$, it holds that,
\begin{equation}\label{eq:asm:1}
c \|y\|^2 \le  y^{\top}\nabla^2 f_i(x)y \le L \|y\|^2
\end{equation}
for all $x \in \real^n$. In other words, for all $i \in \{1,\ldots,m \}$, the eigenvalues of the Hessian of $f_i$ are uniformly bounded from below
 by $c$, and above by $L$, everywhere.
\end{assumption}

This implies the following condition regarding scaling invariance properties appropriate for Newton methods~\cite{deuflhard2011newton}.
\begin{lemma}\label{prop:aff}
For all $\lambda \in \Lambda$, the mapping $x \to \sum_{i=1}^p \lambda_i \nabla^2 f_i(x)$ is affine covariant Lipschitz, i.e.,  meaning that there exists $\omega>0$ such that for all $x, y \in \real^n$, one has
\begin{equation}\label{eq:affinv}
\left\|\left(\sum_{j=1}^m \lambda_j \nabla^2 f_j(x)\right)^{-1} 
\left(\sum_{j=1}^m \lambda_j \nabla f_j(y)-\sum_{j=1}^m \lambda_j \nabla f_j(x)\right) \right\|
\le \omega \|x-y\|.
\end{equation}
\end{lemma}
\begin{proof}
Follows from Assumption~\ref{asm:1} with $\omega=\frac{L}{c}$.
\end{proof}

\begin{lemma}\label{lem:boundeta}
Consider Assumption~\ref{asm:1}. There exists some $\eta$ depending only on properties of $\{f_k(\cdot)\}$ such that $\eta^{(k)}\le \eta\|\lambda^{(k+1)}-\lambda^{(k)}\|$ for all $k$.
\end{lemma}
\begin{proof}
Assumption~\ref{asm:1} which is equivalent to strong regularity~\cite{robinson1980strongly} implies Lipschitz continuity of the mapping $\tilde x(\lambda)$ introduced in Theorem~\ref{th:implicitfun}  for all $\lambda^{(k)}$ by~\cite[Theorem 3.2]{dontchev2013euler}.

Since strong regularity holds, we can apply~\cite[Theorem 5.60]{bonnans2013perturbation}, noting that the problem is unconstrained and 
in that notation $G(x)\equiv 0$ and we let the perturbation $u(t)=\tilde\lambda^{(k)}(t)$ be $\tilde\lambda^{(k)}(t)=t\lambda^{(k+1)}+(1-t)\lambda^{(k)}$.
The application of the Theorem implies that for $\|\lambda^{(k+1)}-\lambda^{(k)}\|$ sufficiently small it holds that there is a unique
$\tilde x(\tilde\lambda^{(k)}(t))$ that is continuously differentiable. 
Since uniform regularity implies solution uniqueness, we can extend $\tilde x(\lambda)$ 
as needed across all $t\lambda^{(k+1)}+(1-t)\lambda^{(k)}$ for $t\in[0,1]$.
And since the grid is compact and again by uniform regularity, 
and there is a unique solution for all $\lambda^{(k)}$, it holds that this applies for all solution paths across the grid and $\tilde x(\lambda)$ is 
globally Lipschitz continuous.

Now we have shown there exists some $C_L$,
\[
\|\tilde x(\lambda^{(k+1)})-\tilde x(\lambda^{(k)})\| \le C_L \|\lambda^{(k+1)}-\lambda^{(k)}\|
\]
and since,
\[
\begin{array}{l}
\tilde x(\lambda^{(k+1)}) = \tilde x(\lambda^{(k)})+\frac{\partial \tilde x^{(k)}}{\partial \lambda}(\lambda^{(k)})^T (\lambda^{(k+1)}-\lambda^{(k)})\\ \qquad\qquad\qquad+
\frac{1}{2}(\lambda^{(k+1)}-\lambda^{(k)})^T\frac{\partial^2 \tilde x^{(k)}}{\partial \lambda^2}(\lambda^{(k)}+t\xi^{(k)} (\lambda^{(k+1)}-\lambda^{(k)}))(\lambda^{(k+1)}-\lambda^{(k)})
\end{array}\]
We have that,
\[
\begin{array}{l}
\eta^{(k)}= \left\|\frac{1}{2}(\lambda^{(k+1)}-\lambda^{(k)})^T\frac{\partial^2 \tilde x^{(k)}}{\partial \lambda^2}(\lambda^{(k)}+t\xi^{(k)} (\lambda^{(k+1)}-\lambda^{(k)}))(\lambda^{(k+1)}-\lambda^{(k)})\right\|
\\ \qquad\le \left\|\tilde x(\lambda^{(k+1)}) - \tilde x(\lambda^{(k)})\right\| + \left\|\frac{\partial \tilde x^{(k)}}{\partial \lambda}(\lambda^{(k)})^T (\lambda^{(k+1)}-\lambda^{(k)})\right\|
\\ \qquad \le C_L \|\lambda^{(k+1)}-\lambda^{(k)}\|+\left\|\sup_\lambda \frac{\partial \tilde x^{(k)}}{\partial \lambda}(\lambda)\right\|\|\lambda^{(k+1)}-\lambda^{(k)}\|:= \eta \|\lambda^{(k+1)}-\lambda^{(k)}\|
\end{array}
\]
\end{proof}

The next result shows that applying  the correction step  will require a number of iterations of the ordinary Newton method of order $\log\log\left(\frac{1}{\epsilon}\right)$ to get an  $\epsilon$-Pareto optimal solution.
\begin{lemma} \label{lm:2}
Let Assumption \ref{asm:1} hold. For a given iteration index $k$, consider $\lambda^{(k)} \in \Lambda$ and $x^{(k)}$ such that
\begin{equation*}\label{epsilon:solution}
\left \| \sum_{j=1}^m \lambda^{(k)}_j \nabla f_i\left(x^{(k)}\right) \right\| \le \epsilon.
\end{equation*}
Let  $\lambda^{(k+1)} \in \Lambda$ such that  
\begin{equation}\label{eq:pcbound}
\left\|\lambda^{(k+1)} - \lambda^{(k)}\right\|_\infty\le \frac{2}{\omega\eta},
\end{equation}
where $\eta$ is as in Lemma~\ref{lem:boundeta} then, the ordinary Newton method with
the starting point $\bar x^{(k+1)}$ (as given by (\ref{predictor:step})) converges 
%towards a solution point  $\left[\bar x^{(k+1)}\right]_{\infty}$ such that $\sum_{j=1}^m \lambda^{(k+1)}_j \nabla f_i(\left[\bar x^{(k+1)}\right]_{\infty})=0$.
and the computational cost of achieving a solution point  $x^{(k+1)}$ such that 
\begin{equation*}
\left \| \sum_{j=1}^m \lambda^{(k+1)}_j \nabla f_i\left(x^{(k+1)}\right) \right\| \le \epsilon.
\end{equation*} is of order
$\log\log\left(\frac{1}{\epsilon}\right)$.
\end{lemma}
\begin{proof}
From (\ref{eq:pcbound}) and~\cite[Theorem 5.2]{deuflhard2011newton}, one conclude the ordinary Newton method with
the starting point $\bar x^{(k+1)}$  converges towards a solution point  $\left[\bar x^{(k+1)}\right]_{\infty}$ such that \\ $\sum_{j=1}^m \lambda^{(k+1)}_j \nabla f_i\left(\left[\bar x^{(k+1)}\right]_{\infty}\right)=0$.

Let  $\left[\bar x^{(k+1)}\right]_j$ be the $j^{th}$ iterate produced by an ordinary Newton method with
the starting point $\bar x^{(k+1)}$. 
We note that by the definition of $\eta^{(k)}$ and using (\ref{eq:pcbound}), one has
\begin{equation}\label{eq:newtconv}
\left\|\left[\bar x^{(k+1)}\right]_0- \left[\bar x^{(k+1)}\right]_{\infty}\right\|\le \eta^{(k)}\le \eta \|\lambda^{(k+1)}-\lambda^{(k)}\| \le \frac{2}{\omega}. 
\end{equation}
In this case, using~\cite[Theorem 2.3]{deuflhard2011newton},  ordinary Newton method with
the starting point $\bar x^{(k+1)}$ converges quadratically, i.e., for each iteration $j$ of the ordinary Newton method, on has
\begin{equation*}\label{eq:newtconvb}
\left\|\left[\bar x^{(k+1)}\right]_{j}- \left[\bar x^{(k+1)}\right]_{\infty}\right\|\le \omega^{2^{j}} \left\|\left[\bar x^{(k+1)}\right]_0- \left[\bar x^{(k+1)}\right]_{\infty}\right\|^{2^j}.
\end{equation*}
Hence, using (\ref{eq:newtconv}), one deduces that 
\begin{equation*}\label{eq:newtconvb}
\left\|\left[\bar x^{(k+1)}\right]_{j}- \left[\bar x^{(k+1)}\right]_{\infty}\right\|\le 2^{2^{j}}.
\end{equation*}
Thus, 
\begin{eqnarray*}
\left \| \sum_{j=1}^m \lambda^{(k)}_j \nabla f_i\left(\left[\bar x^{(k+1)}\right]_{j}\right) \right\| &= & \left \| \sum_{j=1}^m \lambda^{(k)}_j \nabla f_i\left(\left[\bar x^{(k+1)}\right]_{j}\right)- \sum_{j=1}^m \lambda^{(k)}_j \nabla f_i\left(\left[\bar x^{(k+1)}\right]_{\infty}\right) \right \| \\
&\le&  L  \left  \|\left[\bar x^{(k+1)}\right]_{j}- \left[\bar x^{(k+1)}\right]_{\infty} \right \|\le L 2^{2^{j}}.
\end{eqnarray*}
This implies that the computational cost of achieving the desired level of stationarity is 
$O\left(\log\log\left(\frac{1}{\epsilon}\right)\right)$.
\end{proof}

%
%A formal description of the final algorithm is given in Algorithm \ref{alg:PNA}.
%\begin{algorithm}[t]
%	\caption{path-following Newton Algorithm}
%	\label{alg:PNA}
%	\begin{algorithmic}
%\STATE{\textbf{Initialization:}} $x_0 \in\mathbb{R}^n$,  $\lambda^0$. Set $k=0$. $\epsilon>0$.
%\vspace{3mm}
%\STATE{\textbf{Accelerated Gradient}: Starting from $x_0$, run accelerated gradient method to solve 
%\begin{equation*}
%\min_x f^0(x) = \sum_{i=1}^m \lambda_i^0 f_i(x),\, \text{ until } \|\nabla f^{0}(x)\| \le \epsilon.
%\end{equation*}
% We denote by $x_1$ the obtained solution.
%}
%\vspace{3mm}
%		\FOR{Each iteration $k=1,2,3...,N$. ($N$ is the number of points in $\Lambda$).}
%		\vspace{3mm}
%\vspace{3mm}
%\STATE{\textbf{update of $\lambda$:} Set $\lambda_{k+1}$ as a neighbor of $\lambda_k$, (which was not selected yet).}
%\vspace{3mm}
%\STATE{Compute \begin{equation*}
%\bar x_k =x_k+\sum_{i=1}^m \left(\sum_{j=1}^m \lambda_j^k \nabla^2 f(x_k)\right)^{-1} 
%(\lambda_i^{k+1}-\bar \lambda_i^k) \nabla f_i(x_k)
%\end{equation*}}
%\STATE{\textbf{Newton Algorithm}: Starting from $\bar x_k$, run Newton method to solve 
%\begin{equation*}
%\min_x f^{k+1}(x) = \sum_{i=1}^m \lambda_i^{k+1} f_i(x),\, \text{ until } \|\nabla f^{k+1}(x)\| \le \epsilon.
%\end{equation*}
% We denote by $x_{k+1}$ the obtained solution.
%}
%\ENDFOR
%\end{algorithmic}
%\end{algorithm}

Thus the complexity of Algorithm \ref{alg:PNA} is just of the order of complexity for solving a standalone strongly convex problem 
(i.e., computing $x^{(0)}$) added with
$\frac{1}{d^{m-1}}$ multiplied by the cost of a predictor and a Newton step. We formalize this with the following theorem,

\begin{theorem} 
Let Assumption \ref{asm:1} hold.
Define $N_{\epsilon}$ %= \mathcal{O}\left(log \left(\frac{1}{\epsilon}\right) \right) $ 
 to be the number of iterations required to obtain $x^{(0)}$, 
a point that has distance at most $\epsilon$ from the
optimal point corresponding to~\eqref{eq:scalarmoo} at $\lambda^{(0)}$.  Assume that the maximum width between any two neighbors on the grid $\Lambda$ is,
\begin{equation}\label{cond:d}
d \le \min \left(\frac{2}{\omega\eta} , \bar d \right),
\end{equation}
with $\bar d$ the minimal desired distance between lattice points.

Then,  the overall iteration complexity of Algorithm \ref{alg:PNA} is  $$N_{\epsilon}+\mathcal{O}\left( \left\lfloor{\frac{1}{d}}\right\rfloor^{m-1}\log\log\left(\frac{1}{\epsilon}\right) \right).$$
%where,
%\[
%d\le \min\left\{\left(c\epsilon\right)^{-1/2},\omega^{-1/2},\bar d\right\}
%\]
%with $\bar d$ the minimal desired distance between lattice points.
\end{theorem}
\begin{proof} 
%The first term $\sqrt{\kappa} \log(1/\epsilon)$ is the complexity of finding a solution to a strongly convex problem, 
First, note that, for each iteration $k$ of  Algorithm \ref{alg:PNA}, 
the complexity of the predictor step is constant as its computational cost does not depend on $\epsilon$.
For the corrector Newton step, since one has 
\begin{equation*}
\|\lambda^{(k+1)} - \lambda^{(k)}\|_\infty\le d \le \frac{2}{\omega\eta},
\end{equation*}
Lemma \ref{lm:2} implies that  complexity of such step is $O\left(\log\log\left(\frac{1}{\epsilon}\right)\right)$. The proof is thus completed  since the total number of lattice points in the grid $\Lambda$ is at most $\left\lfloor{\frac{1}{d}}\right\rfloor^{m-1}$.
\end{proof}

Note that by using a gradient solver, the first term $N_{\epsilon}$ is of order $\log(1/\epsilon)$. 
Hence, one can see that the complexity is generally favorable compared to the naive method of solving the strongly convex problem
at every grid point separately, as $\log\log\left(\frac{1}{\epsilon}\right) \ll  \log(1/\epsilon)$ for small $\epsilon$.
%The opposite would occur only if $\epsilon$ is large and $\frac{\omega\epsilon}{2c}$ is large.

%, or the affine covariance constant is disfavorable to Newton's method performing well locally, i.e., $\omega^{1/2}$ is large.
\begin{remark}
Note that both the naive method of solving every problem across the grid points and path-following are both
about equally parallelizeable with perfect speedup as long as the number of grid points is larger
than the number of processors.
We  can split the grid into disjoint components, and each processor finds one point in its part of the convex hull of allowable $\{\lambda_i\}$ and proceeds to pathfollow across the grid component assigned to it.
\end{remark}

%\begin{remark}
%It does seem strange that the bound requires the steps to be less than something proportional to $1/\epsilon$. This is due
%to the fact that a more precise solution means that the derivative is near zero, and thus will allow for a longer
%traversal of the space to reach a point such that this derivative function is no longer mostly flat. 
%This is of course a loose bound, as we expect $d$ to be fairly small in practice.
%\end{remark}

%rate etc.
\section{Numerical Illustration}
\label{sec:Experiments}

To show the numerical efficiency of our approach compared to the naive method (which corresponds to the Gradient Descent method applied sequentially to the set of problems  $(\ref{eq:scalarmoo})$ by varying $\lambda$), we will show the potential of the proposed approach on a very simple problem \cite{Huband-2006}, defined by 
$$f(x) = \left[(x_1-1)^2+(x_1-x_2)^2, (x_2-3)^2+(x_1-x_2)^2\right]^\top$$
Since we have two objective functions, the vector $\lambda$ has two components $\lambda_1$ and $\lambda_2$ where $\lambda_1 + \lambda_2 = 1$. In our experiment, we descritize $\lambda$ in a uniform grid with  a grid stepsize $d$ (the desired distance between the lattice points). 

In our Matlab illustration, we will call \textbf{Multi-GD} the naive method and \textbf{GD+Newton Pathfollowing} the implementation of our Algorithm \ref{alg:PNA} (where we used Gradient Descent method to find the first Pareto optimal point and then apply the Newton path-following procedure). For the Gradient Descent method, we used a random initial point $x_0$ and a stepsize equal to $1/\lambda_{\max}$ where $\lambda_{\max}$ is the maximum eigenvalue for the Hessians of $f_1$ and $f_2$.  We stopped the methods when the norm of the gradient is less than $10^{-7}$.

% or the maximum number of iterations exceeds $10^5$.
The obtained results are depicted in Figure \ref{fig:pareto}, one can see that both methods are able to find a similar Pareto Front (independently of the value of the grid stepsize $d$). In term of the elapsed CPU time to find the front, our proposed algorithm is shown to be faster than the naive method. In particular, one can see that,for different values of $d$, the method \textbf{GD+Newton Pathfollowing} runs $10$ times faster than the \textbf{Multi-GD} method.  We conducted other experiments (not reported here) on many toy problems, and in all the experiments, in term of running time our method was outperforming the naive method, while finding essentially the same front.
%   
% \begin{figure}[h]
%\centering
%{\includegraphics[scale=0.9]{./images/front_pareto_d_01.eps}}
%\caption{$\epsilon$-Pareto Front using naive method and our Algorithm \label{fig:ParetoFront}} %during the application of Algorithm~\ref{alg:LM}
%\end{figure}

\begin{figure}[!h]
\centering
\subfigure[Pareto Front with $d=10^{-1}$.]{
\includegraphics[scale=0.4]{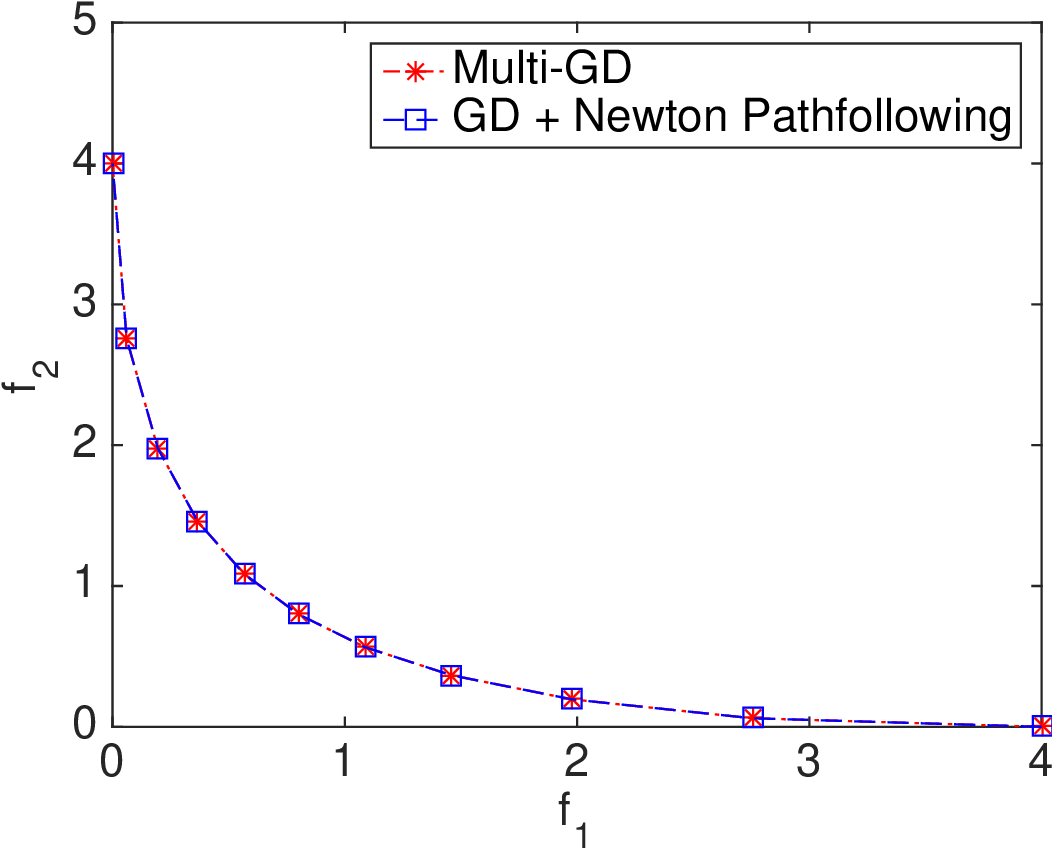}
}
\subfigure[Pareto Front with $d=10^{-2}$.]{
\includegraphics[scale=0.4]{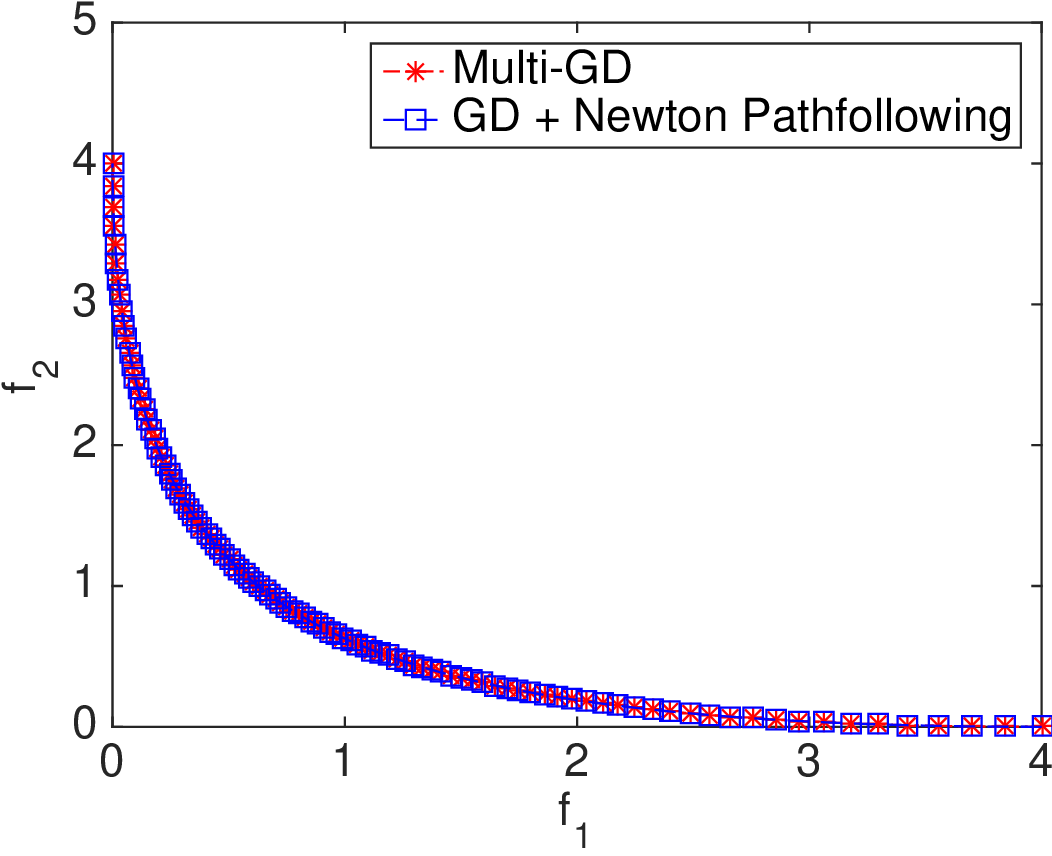}
}

\subfigure[Pareto Front with $d=10^{-3}$.]{
\includegraphics[scale=0.4]{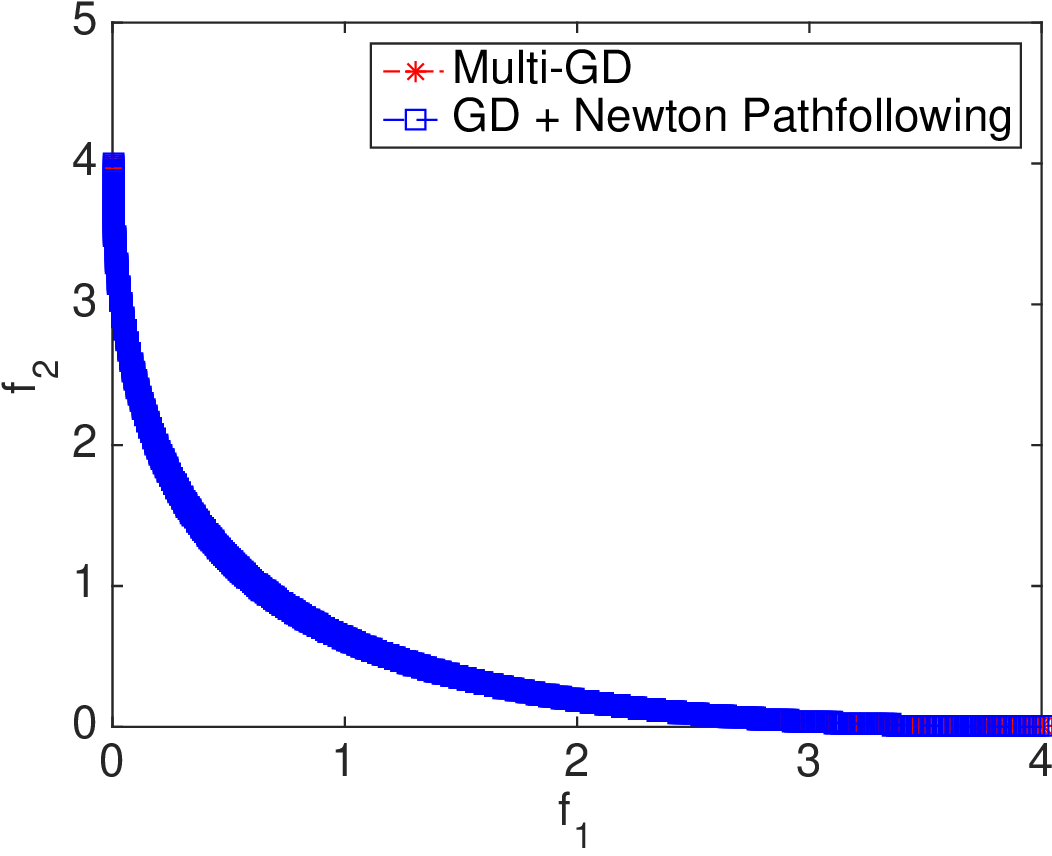}
}
\subfigure[CPU time.]{
\includegraphics[scale=0.4]{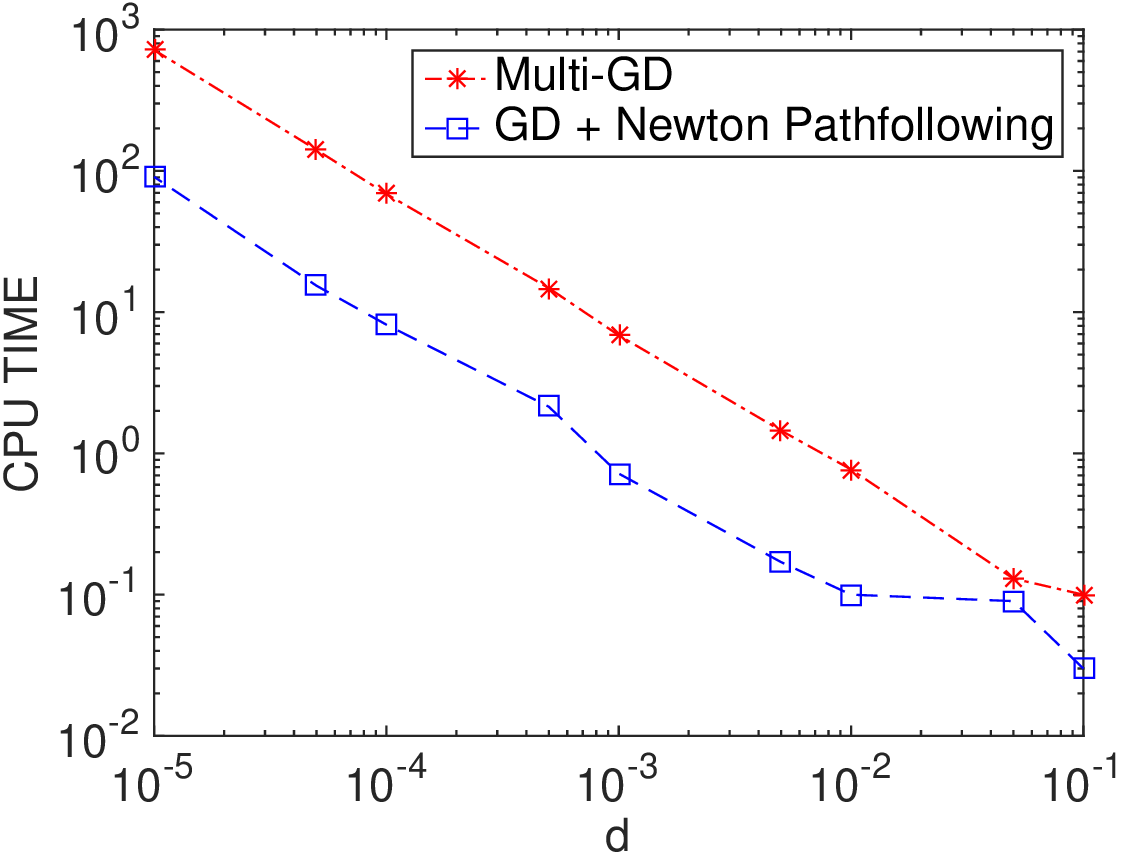}
}
\caption{Pareto Front and CPU time comparison, using \textbf{Multi-GD} and \textbf{GD+ Newton Path-following}, for different values of $d$.}
\label{fig:pareto}
\end{figure}

%Consider, however, the work~\cite{dontchev2013euler} which obtains a regular rate of convergence of performing
%path-following for solving a parametric optimization given a solution at one value of the parameter and moving onto
%the solution at another. In this work a bound was found to the effect of $\|u(t)-u_0(t)\|=O(h^4)$ where $h$
%defines the discretization of the path-following, and the error between the computed solution and the actual
%one is sought to be minimized. 
\section{Conclusion}
\label{sec:Conclusion}

In this note we studied the complexity of a class of strongly convex multi objective optimization problems. We observed that such a problem
is not uniquely defined, given different criteria of what it means to solve multiobjective optimization. 
Picking the most context-independent
criterion -- finding the set of all Pareto optimal points on a front, we demonstrated that in most cases, finding
the solution of one scalarized problem and path-following across the grid to obtain the others is superior to
finding the solution of every problem independently.

\bibliographystyle{plain}
\bibliography{refs}

\end{document}